\documentclass [a4paper] {amsart}
 \usepackage{amsthm,amssymb,amsmath,amsopn,amsfonts}
 \usepackage{graphicx}
 \usepackage{indentfirst}
 \usepackage{geometry}
 \usepackage{mathtools}
 \usepackage{float}
 \usepackage{todonotes}

\theoremstyle{plain}

\newtheorem*{theorem*}{Theorem}

\DeclareMathOperator{\supp}{supp}

\usepackage[english]{babel}
\usepackage[T1]{fontenc}
\usepackage[utf8]{inputenc}
\usepackage{enumerate}
\usepackage{csquotes}

\newtheorem*{thm*}{Theorem} 
\newtheorem{thm}{Theorem} 
\newtheorem*{lemma*}{Lemma}
\newtheorem{lemma}{Lemma}
\newtheorem*{prop*}{Proposition}

\newtheorem*{corl*}{Corollary}
\newtheorem{corl}{Corollary}

\theoremstyle{definition}

\newtheorem*{obs*}{Observation}

\newtheorem*{remark*}{Remark}

\theoremstyle{definition}

\newcommand{\reals}{\mathbb{R}}

\begin{document}

\title[]{Oscillation of functions in the H\"{o}lder class}

\author[P.~Mozolyako]{Pavel Mozolyako}
\address{Universit\`{a} di Bologna, Dipartimento di Matematica, Piazza di Porta S. Donato, 40126 Bologna (BO)}
\email{pavel.mozolyako@unibo.it}
\thanks{The first author is supported  by the joint grant of Russian Foundation for Basic Research (project 17-51-150005-NCNI-a) and CNRS, France (project PRC CNRS/RFBR 2017-2019 "Noyaux reproduisants dans des espaces de Hilbert de fonctions analytiques"}
\author[A.~Nicolau]{Artur Nicolau}
\address{Universitat Aut\`{o}noma de Barcelona, Departament de Matem\`{a}tiques,\\
08193 Barcelona}
\email{artur@mat.uab.cat }
\thanks{The second author is supported in part by the Generalitat de Catalunya (grant 2017 SGR 395) and the Spanish Ministerio de Ciencia e Innovaci\'on (projects  MTM2017-85666-P and Maria de Maeztu Unit of Excellence MDM-2014-0445).}

\date{}


\begin{abstract}
    We study the size of the set of points where the $\alpha$-divided difference of a function in the H\"{o}lder class $\Lambda_\alpha$ is bounded below by a fixed positive constant. Our results are obtained from their discrete analogues which can be stated in the language of dyadic martingales. Our main technical result in this setting is a sharp estimate of the Hausdorff measure of the set of points where a dyadic martingale with bounded increments has maximal growth.    
\end{abstract}

\maketitle

    \section{Introduction}
    
    For $0< \alpha <1$ let $\Lambda_\alpha (\mathbb{R})$ be the H\"{o}lder class of functions $f: \mathbb{R} \rightarrow \mathbb{R} $ such that there exists a constant $C=C(f) >0$ with $|f(x) - f(y)| \leq C |x-y|^\alpha$ for any $x, y \in \mathbb{R}$. The infimum of such constants C is denoted by $\|f\|_\alpha$. For $b>1$, G.H. Hardy proved in \cite{Hardy} that the Weierstrass function
    \begin{equation}
        \label{eq:Weierstrass}
        f_{b, \alpha} (x) = \sum_{n=0}^\infty b^{- n \alpha} \cos (b^n x) ,  \quad  x \in \mathbb{R},  
    \end{equation}
is in $\Lambda_\alpha (\mathbb{R})$ and exhibits the extreme behavior
\begin{equation*}
    \limsup_{h \to 0} \frac{|f_{b, \alpha} (x + h) -  f_{b, \alpha} (x)|}{|h|^\alpha} >0
  \end{equation*}
at any point $x \in \mathbb{R} $. The main purpose of this note is to study the $\alpha$-divided differences defined as
\begin{equation}
        \label{eq:alphadivided}
        \Delta_{\alpha} (f) (x,h) =  \frac{f (x + h) -  f (x)}{|h|^{\alpha}} , 
\end{equation}
for functions $f \in \Lambda_\alpha (\mathbb{R})$. Let $\sigma$ be the standard Haar measure of $(0, \infty)$ defined as
$$
\sigma (E) = \int_E \frac{dh}{h} ,\quad E \subset (0, \infty) . 
$$
Hence $\sigma[h,1] = \log h^{-1}$, $0<h <1$. For $k=0,1,2,\dotsc$, let $\mathcal{D}_k$ be the collection of dyadic intervals in $\mathbb{R}$ of
generation or rank~$k$ of the form $[j2^{-k}, (j+1)
2^{-k})$ where $j$ is an integer. Let
$\mathcal{D}=\cup \mathcal{D}_k $ be
the collection of all dyadic intervals. For $0<s \leq 1$ let $H^s (E)$ denote the dyadic Hausdorff measure of the set $E \subset \reals$, that is, $H^s (E) = \lim_{\delta \to 0 } H^s_\delta (E)$, with 
\begin{equation*}
    H^s_\delta (E) = \inf \sum |I_j|^s , 
\end{equation*}
where the infimum is taken over all collections of dyadic intervals $\{I_j \}$ of length $|I_j| < \delta$ with $E \subset \bigcup I_j$. The Hausdorff dimension of $E$, denoted by $\dim (E)$, is the infimum of the indices $s>0$ such that $H^s (E) < \infty$. 
Let $f \in \Lambda_\alpha  (\reals) $. Roughly speaking, our first result says that at almost every point $x \in \reals$ either $ \Delta_{\alpha} (f) (x,h)$ is small for a significant set of scales $h>0$ or $ \Delta_{\alpha} (f) (x,h)$ oscillates around the origin infinitely often when $h$ tends to 0. 
\begin{thm}
\label{mean}
Let $0< \alpha < 1$ and $f \in \Lambda_\alpha (\reals)$. At almost every point $x \in \reals$ such that there exists a constant $\delta = \delta (x) >0$ with 
\begin{equation}\label{upper}
 \limsup_{h \to 0^+} \frac{\sigma\{t \in [h,1]: \Delta_{\alpha} (f) (x,t) > \delta \}}{\log {h}^{-1}} > 0, 
\end{equation} 
there exists a constant $c=c(x) >0$ such that
\begin{equation}\label{lower}
 \limsup_{h \to 0^+} \frac{\sigma\{t \in [h,1]: \Delta_{\alpha} (f) (x,t)< - c  \}}{\log {h}^{-1}} > 0.
\end{equation} 
\end{thm}   
For any $b>1$ and $0< \alpha <1$, the Weierstrass function $f_{b,\alpha}$ defined in \eqref{eq:Weierstrass} satisfies condition \eqref{upper} (and also \eqref{lower}) at any point $x \in \mathbb{R}$ for certain uniform constants $\delta = \delta (b, \alpha)$ and $c = c (b, \alpha)$. This is discussed after the proof of Theorem \ref{mean} at the end of Section 3. Observe that in both the assumption and conclusion of Theorem  \ref{mean}, one uses $\sigma$ to measure the set of scales where $\Delta_{\alpha} (f) (x,t)$ is not small.  Our next result shows that this is essential. 

\begin{thm}
\label{t:contraexemple}
Let $0<\alpha<1$. Then there exists a function $f \in \Lambda_\alpha (\reals)$ such that at almost every $x \in\reals$ one has
\[
\limsup_{ h \rightarrow 0^+} \Delta_{\alpha} (f) (x,h) >0
\]
and
\[
\liminf_{ h \rightarrow 0^+} \Delta_{\alpha} (f) (x,h) = 0 . \]
\end{thm}

The $\alpha$-divided differences $\Delta_{\alpha} (f) (x,h)$ of a function $f \in \Lambda_\alpha  (\reals) $ may oscillate as $h$ tends to $0$ at every point $x \in \reals$. However our next result says that the set of $x \in \reals$ where $\Delta_{\alpha} (f) (x,h)$ is bounded below by a positive constant as $h \to 0$, is always small in the sense that one can estimate its Hausdorff dimension.The statement of our results use the following entropy
\begin{equation}\label{entropy}
\Phi (\eta) = \frac{1+\eta}{2}\log_2\left(\frac{2}{1+\eta}\right)+\frac{1-\eta}{2}\log_2\left(\frac{2}{1-\eta}\right), \quad 0 < \eta < 1.
\end{equation} 
The entropy function $\Phi$ appears naturally in the study of the dimension of various sets and measures appearing in dynamical and probabilistic contexts. See the survey \cite{Heurteaux} and the references there. 

\begin{thm}
        \label{thm:continuous}
Let $0< \alpha < 1$ and $f \in \Lambda_\alpha  (\reals)$ with $\|f \|_{\alpha} \leq 1$. 

(a) For $0< \gamma < 1$, consider the set $G(\gamma)$ of points $x \in \reals$ such that
\begin{equation*}
    \limsup_{h \to 0^+} \frac{\sigma\{t \in [h,1]: \Delta_{\alpha} (f) (x,t) > \gamma \}}{\log h^{-1}} =1 . 
\end{equation*}
Then  $\dim G(\gamma) \leq \Phi (  \gamma )$.

(b)  For $0<\gamma<1$ consider the set $F(\gamma)$ of points $x \in \reals$ such that 
\[
 \liminf_{h\rightarrow0^+} \Delta_{\alpha} (f) (x,h) >\gamma.
\]
Then $\dim F(\gamma) \leq \Phi(  \gamma )$.
\end{thm}

We first prove discrete versions of our results which will be stated in the language of dyadic martingales. For $x\in\mathbb{R}$ and $k \geq 0$
let $I_k (x)$ be the unique interval in $\mathcal{D}_k
$ which contains $x$.  Also $|E|$ denotes the Lebesgue measure
of the measurable set $E \subset \mathbb{R}$. A dyadic
martingale is a sequence of locally integrable
functions~$S=\{ S_k \}$ such that for any $k=0,1,2,\dotsc$,
the function $S_k$ is measurable with respect to the
$\sigma$-algebra ${\mathcal{F}}_k$ generated by $\mathcal{D}_k $ and
the conditional expectation of $S_{k+1}$ respect to $\mathcal{F}_{k}$
is $S_{k}$. In other words, for any $k=0,1,2,\dotsc$, the
function $ S_k $ is constant in each dyadic interval
of $\mathcal{D}_k$ and
\begin{equation}
\label{cancellation}
\int_I \left(S_{k+1}(x)-S_k (x)\right)\,dx=0
\end{equation}
for any $I\in \mathcal{D}_k$. The main idea of the proof of our results is to consider the dyadic martingale $\{S_n \}$ defined as 
\begin{equation}\label{quocientsincrementals}
S_n (x) = 2^{n} (f(b) - f(a)) , \quad  x \in [a,b) \in \mathcal{D}_n 
\end{equation}
and to establish discrete versions of our results. These discrete versions are based on the following estimate which has independent interest.    

Let $\{S_n \}$ be a dyadic martingale with bounded increments, that is, satisfying  $\|\{S_n \}\|_* = \sup_n \|S_{n+1} - S_n \|_\infty < \infty$. In the seminal paper \cite{Makarov}, Makarov used the subclass of Bloch martingales to study the boundary behavior of functions in the Bloch space and the metrical properties of harmonic measure in simply connected domains of the complex plane. It is clear that $\|S_n - S_0 \|_\infty \leq \|\{S_n \}\|_* n $. Our next result provides an estimate of the size of the set of points $x \in \reals$  where $S_n (x)$ grows as a proportion of $n$. 

\begin{thm}\label{t:th2}
Let $\{S_n\}$ be a dyadic martingale with $\|S_{n+1}-S_n\|_{\infty}\leq1$ for $n=1,2,\ldots$. For $0<\eta<1$ consider the set
\begin{equation}\label{e:21}
E(\eta)=\{x\in \reals :\limsup_{n\rightarrow\infty}\frac{S_n(x)}{n} \geq \eta\} .
\end{equation}
Then $H^{\Phi (\eta)} (E(\eta)) \leq 1$ and consequently $\dim (E(\eta)) \leq \Phi (\eta)$. 
\end{thm}

The result is sharp in the sense that there are dyadic  martingales $\{S_n \}$ with $\|\{ S_n \}\|_* = 1 $ such that $\dim E (\gamma) = \Phi (\gamma)$. Actually we will show that  these examples correspond to the classical result of Besicovitch \cite{Besicovitch} (generalized later by Eggleston in \cite{Eggleston}) on the Hausdorff dimension of sets of real numbers which are defined by their digital expansions. Observe that if $f \in \Lambda_\alpha (\reals)$, the martingale $\{S_n \}$ defined in \eqref{quocientsincrementals} satisfies $ \sup_n 2^{- n (1 - \alpha)} \|S_n \|_\infty \leq \|f\|_\alpha $. Fix $0< \beta <1 $. Theorem \ref{t:th2} is used to study the size of the set of points of maximal growth of dyadic martingales $\{T_n \}$ satisfying the growth condition $\sup_n 2^{-n \beta} \|T_n \|_\infty <  \infty$. In particular we obtain discrete analogues of Theorems \ref{mean} and  \ref{thm:continuous} which are collected in the following statement.

\begin{corl}
        \label{corl:continuous}
Let $0< \beta < 1$ and let $\{T_n \}$ be a dyadic martingale  with $\sup_n 2^{- n \beta} \|T_n - T_{n-1} \|_{\infty} \leq 1$. 
\begin{enumerate}[(a)]
            \item For $0< \gamma < 1$, consider the set $G(\gamma)$ of points $x \in \reals$ such that
\begin{equation}
    \label{definicio G1}
    \limsup_{N \to \infty} \frac{1}{N} \# \{1 \leq k \leq N : 2^{- k \beta} T_k (x) \geq \gamma  \} =1 . 
\end{equation}
Then $H^{\Phi ( \gamma (1- 2^{ - \beta})) } ( G(\gamma) )\leq 1$  and consequently $dim G(\gamma) \leq \Phi( \gamma (1- 2^{ - \beta}))$.
     
        \item
         For $0<\gamma<1$ consider the set $F(\gamma)$ of points  $x \in \reals$ such that
\[
\liminf_{k \to \infty}  2^{- k \beta} T_k (x) \geq \gamma .
\]
Then $H^{\Phi ( \gamma (1- 2^{ - \beta})) } ( F(\gamma) ) \leq 1$  and consequently $dim F(\gamma) \leq \Phi( \gamma (1- 2^{ - \beta}))$.
\item
At almost every point $x \in \reals$ such that there exists a constant $\delta = \delta (x) >0$ with 
$$
 \limsup_{N \to \infty} \frac{1}{N} \#  \{1 \leq k \leq N : 2^{- k \beta} T_k (x) > \delta  \} > 0 ,  
$$
there exists a constant $c=c(x) >0$ such that
$$
 \limsup_{N \to \infty} \frac{1}{N} \#  \{1 \leq k \leq N : 2^{- k \beta} T_k (x) < -c   \} > 0 
$$
        \end{enumerate}
\end{corl}

It is worth mentioning that the strategy of obtaining continuous results from their dyadic analogues has certain limitations. Fix $0 < \beta < 1$ and let $\{T_n \}$ be a dyadic martingale such that $\sup_n 2^{- n \beta} \|T_n \|_\infty < \infty$. Fern?ndez, Heinonen and Llorente proved the following $0-1$ Law: for any interval $I$ either  $\{T_n (x)\}$ converges at a set of points $x \in I$ of positive length or there exists a constant $C>0$ such that 
$$
M_{1- \beta} (\{x \in I : \lim_{n \to \infty } T_n (x) = + \infty  \}) > C |I|^{1- \beta}. 
$$
Here $M_{1 - \beta}$ denotes the $(1-\beta)$-Hausdorff content. See \cite{FenandezHeinonenLlorente}. However the continuous analogue of this result fails. Actually a H\"{o}lder continuous function may oscillate wildly around every point. 

\begin{thm}
\label{p:51}
Let $0 < \alpha < 1$. Then there exists a function $f\in \Lambda_{\alpha} (\reals)$ and a constant $C>0$ such that for any point $x\in \mathbb{R}$ there exist two  sequences $\{h_k\}_{k=1}^{\infty},\; \{h'_k\}_{k=1}^{\infty}$ of positive numbers, converging to zero, such that
\begin{equation}\label{e:551}
\begin{split}
&\limsup_{k\rightarrow\infty}\left|\frac{f(x+h'_k)-f(x)}{h'_k}\right| \leq 1, \\
&\liminf_{k\rightarrow\infty} \frac{|f(x+h_k)-f(x)|}{|h_k|^{\alpha}}> C.
\end{split}
\end{equation}
\end{thm} 

The paper is organized as follows. Section 2 contains the proof of Theorem \ref{t:th2} and Corollary \ref{corl:continuous}. In Section 3 we consider the accumulated $\alpha$-divided difference and deduce Theorems  \ref{mean} and \ref{thm:continuous} from Theorem \ref{t:th2}. Section 4 contains the proof of Theorem \ref{t:contraexemple}. Section 5 is devoted to the proof of Theorem \ref{p:51}. Finally Section 6 is devoted to another application of Theorem \ref{t:th2} to estimate the size of the set where a function in the Bloch space has maximal growth.

\section{Proof of Theorem \ref{t:th2}}
 The proof of Theorem \ref{t:th2} uses the following two  elementary auxiliary results. The first one is certainly well known but its short proof is included for the sake of completeness. 
 
\begin{lemma}\label{l:31}
Let $\mathcal{A}$ be a collection of dyadic intervals of the unit interval $[0,1]$. Let  $\mu$ be a finite positive Borel measure on $[0,1]$ such that $\mu (I) \geq |I|^s$ for any $I \in \mathcal{A}$. Let $E$ be the set of points which belong to infinitely many distinct intervals of the collection $\mathcal{A}$. Then $H^s (E) \leq  \mu ([0,1]) $. In particular $dim E \leq s$.
\end{lemma}
\begin{proof}
We can assume that $\mathcal{A}$ has infinitely many different dyadic intervals. Fixed $\delta >0$, let $\mathcal{A} (\delta)$ be the collection of maximal dyadic intervals of $\mathcal{A}$ of length smaller than $\delta$. Observe that $E$ is contained in the union of the intervals of $\mathcal{A} (\delta)$. Hence
\[
H^s_{\delta} (E) \leq \sum_{I \in \mathcal{A} (\delta) } |I|^s . 
\]
By maximality the intervals of $\mathcal{A} (\delta)$ are pairwise disjoint. Hence the assumption gives that 
\[
 \sum_{I \in \mathcal{A} (\delta) } |I|^s \leq \mu ([0,1]),
\]
which finishes the proof. 
\end{proof}

\begin{lemma}\label{l:32}
Let $0 < \eta < 1$ and let $x_k \in \mathbb{R}$ with $|x_k | \leq 1$, $k=1,2,\ldots , N $. Assume that
\begin{equation}\label{e:321}
\sum_{k=1}^N x_k \geq \eta N . 
\end{equation}
Then 
$$
\prod_{k=1}^{N} \frac{1 + \eta x_k}{2} \geq 2^{- N \Phi (\eta)}
$$
\end{lemma}

It is worth mentioning that if $N_0 = N (1+ \eta)/2$ is an integer, the choice $x_k = 1$ for $1 \leq k \leq N_0$ and  $x_k = -1$ for $N_0 < k \leq N$, gives that
$$
\prod_{k=1}^{N} \frac{1 + \eta x_k}{2} =  2^{- N \Phi (\eta)}
$$

\begin{proof}
The convexity of the function $f(t) = \log_2 2(1 + \eta t)^{-1}$, $|t|<1$, gives that 
$$
\log_2 \frac{2}{1+ \eta x_k} \leq \frac {1+ x_k}{2} \log_2 \frac {2}{1+ \eta } + \frac {1- x_k}{2} \log_2 \frac {2}{1-  \eta }, \quad  k=1, \ldots , N. 
$$
Consider $X = \sum_{k=1}^N x_k$. Adding over $k$ in the previous estimate one obtains
$$
\sum_{k=1}^N \log_2 \frac{2}{1+ \eta x_k} \leq \frac{N+ X}{2} \log_2 \frac {2}{1+ \eta } + \frac {N- X}{2} \log_2  \frac {2}{1-  \eta }. 
$$
Applying (\ref{e:321}), one deduces 
$$
\sum_{k=1}^N \log_2 \frac{2}{1+ \eta x_k} \leq N \frac{1+ \eta}{2} \log_2 \frac {2}{1+ \eta } + N \frac {1- \eta}{2} \log_2 \frac {2}{1-  \eta } = N \Phi (\eta). 
$$
\end{proof}
We are now ready to prove Theorem \ref{t:th2}. 

\begin{proof}[Proof of Theorem \ref{t:th2}]

We can assume that $S_0 = 0$. If $I$ is a dyadic interval of length $2^{-n}$ we denote by $S(I)$ the constant value of $S_n$ at $I$, that is, $S(I) = S_n (x)$, $x \in I$. Consider the collection $\mathcal{A}$ of dyadic intervals $I$ such that $S(I) \geq \eta \log_2 (1/|I|)$. By Lemma \ref{l:31}, it is sufficient to construct a positive Borel measure $\mu$ on $[0,1]$ such that 
\begin{equation}\label{e:322}
\mu (I) \geq |I|^{\Phi (\eta)}, I \in \mathcal{A}. 
\end{equation}
 The measure $\mu$ is constructed inductively by prescribing its mass on every dyadic interval. Define $\mu([0,1]) = 1$. Let $I$ be a dyadic interval and assume $\mu(I)$ has been defined. Let $I_-$ and $I_+$ be the two dyadic intervals contained in $I$ of length $|I_-|=|I_+|= |I|/2$. Denote by $\Delta S (I_-)$ (respectively $\Delta S (I_+)$) the jump of the martingale at $I_-$ (respectively $I_+$), that is, $\Delta S (I_-) = S(I_- ) - S(I)$ ( respectively  $\Delta S (I_+) = S(I_+ ) - S(I)$). Define
$$
\mu (I_- ) = \frac{1 + \eta \Delta S (I_-) }{2} \mu (I) \quad \text { and } \quad  \mu (I_+ ) = \frac{1 + \eta \Delta S (I_+) }{2} \mu (I) \, .
$$
This defines a probability measure $\mu$ on $[0,1]$. Given a dyadic interval $I$ of length $2^{-N}$ and $0 \leq k \leq N$, let $I_k$ be the unique dyadic interval of length $2^{-k}$ which contains $I$. Then 
$$
\mu (I) = \prod_{k=1}^N \frac{1 + \eta (S(I_k) - S(I_{k-1}))}{2} \, .
$$
Observe that
$$
S (I)  =  \sum_{k=1}^N (S(I_k) - S(I_{k-1})).  
$$
Then the estimate (\ref{e:322}) follows from Lemma \ref{l:32}. 
 \end{proof}

Let us now discuss the sharpness in Theorem \ref{t:th2}. Let $\{x_n(x)\},\; n=1,2,\dots$ be the sequence of binary digits of the point $x\in [0,1]$. Consider the dyadic martingale $\{S_n\}$ defined as $S_0 \equiv 0$, and 
\begin{equation}
\label{binary}
S_n(x) := 2 \sum_{k=1}^{n}\left(x_k(x) - \frac12\right),\quad n=1,2,\dots.
\end{equation}
So, if $k_n (x)$ is the number of ones in the first $n$ binary digits of $x$, then $S_n (x) = 2 k_n (x) - n$. For $0< \eta < 1$ we have
\begin{equation*}
 \left\{x \in [0,1] :\; \limsup_{n\rightarrow\infty}\frac{S_n(x)}{n} =  \eta\right\} =  \left\{x \in [0,1] :\; \limsup_{n\rightarrow\infty}\frac{k_n(x)}{n} = \frac{1 + \eta}{2} \right\}
\end{equation*}
and it is a classical result of Besicovitch that this set has Hausdorff dimension $\Phi (\eta)$. See \cite{Besicovitch}. Actually it is also known that
\begin{equation*}
   \dim \left\{x:\; \limsup_{n\rightarrow\infty}\frac{S_n(x)}{n}\geq \eta\right\} = \dim \left\{x:\; \liminf_{n\rightarrow\infty}\frac{S_n(x)}{n}\geq\eta\right\}  = \Phi (\eta).
\end{equation*}
See \cite{CCC} and the references there. 

Fix $0 < \beta < 1$. We will be interested in dyadic martingales $\{T_n \}$ satisfying the growth condition $\|\{T_n \}\|_\beta = \sup_n 2^{- n \beta} \|T_n \|_\infty < \infty$. However it will be more convenient to use the equivalent condition 
\begin{equation}\label{equivalentnorm}
\|\{T_n \}\|_{\beta, *} = \sup_n 2^{- n \beta} \|T_n - T_{n-1} \|_\infty < \infty. 
\end{equation}
Observe that $\|\{T_n \}\|_{\beta, *} \leq (1+ 2^{-\beta}) \|\{T_n \}\|_\beta$ and writing $T_n - T_0$ as sum of $T_k - T_{k-1}$, we deduce that $\|\{T_n - T_0 \}\|_\beta \leq 2^\beta (2^\beta - 1)^{-1} \|\{T_n \}\|_{\beta, *} $. We start with the following consequence of Theorem \ref{t:th2}.  

\begin{corl}
\label{corl1}
Fix $0< \beta < 1$. Let $\{T_n\}$ be a dyadic martingale with $\|\{T_n \}\|_{\beta, *} \leq 1$. 
For $0< \gamma < 1$, consider the set $H(\gamma)$ of points $x \in \reals$ such that
\begin{equation}
    \label{definicio F}
    \limsup_{N \to \infty} \frac{1}{N} \sum_{k=1}^N 2^{-k \beta} T_k (x) \geq \gamma . 
\end{equation}
Then $H^{\Phi (\gamma (1- 2^{-\beta}))} (H(\gamma)) \leq 1$ and consequently $\dim H(\gamma) \leq \Phi (\gamma (1- 2^{-\beta}))$. 
\end{corl}

\begin{proof}
Consider the dyadic martingale $\{S_n \}$ defined by $S_0 = 0$ and 
\begin{equation}\label{e:T}
S_n (x) = \sum_{k=1}^n 2^{-k \beta} (T_k (x) - T_{k-1} (x)), \quad x \in \reals , n=1,2, \ldots
\end{equation}
Summation by parts gives that 
\begin{equation}\label{parts}
S_n (x) = (1- 2^{- \beta}) \sum_{k=1}^{n-1} 2^{- k \beta} T_k (x) + 2^{-n \beta} T_n (x) - 2^{- \beta} T_0 (x) , \quad n \geq 1. \end{equation}
Hence if $x \in H(\gamma) $ we have $\limsup n^{-1} S_n (x) \geq \gamma (1- 2^{- \beta})$. Since $\|S_n - S_{n-1}\|_\infty \leq 1$ for $n=1,2,\ldots$, Theorem \ref{t:th2} gives that $H^{\Phi (\gamma (1- 2^{-\beta}))} (H(\gamma)) \leq 1 $. 
\end{proof}

Let us now discuss the sharpness in Corollary \ref{corl1}. Actually we will show that there exists a dyadic martingale $\{T_n \}$ with $\|\{T_n \}\|_{\beta , * } =1$ for which $\dim H(\gamma) = \Phi (\gamma (1- 2^{- \beta}))$. If $\{S_n \}$ is the dyadic martingale defined in \eqref{binary}, consider the dyadic martingale $\{T_n \}$ defined by $T_0 =0$ and 
$$
T_n = \sum_{k=1}^n 2^{ k \beta} (S_k - S_{k-1}) , \quad n=1,2,\ldots
$$
Hence $T_n - T_{n-1} = 2^{n \beta} (S_n - S_{n-1})$, $n=1,2, \ldots$ and 
$$
S_n - S_0 = \sum_{k=1}^n 2^{- k \beta} (T_k - T_{k-1}) . 
$$
By \eqref{parts}, $x \in H(\gamma)$ if and only if 
$$
\limsup_{n \to \infty} \frac{S_n (x)}{n} \geq \gamma (1- 2^{-\beta}). 
$$ 
So, the classical result \cite{Besicovitch} mentioned above, give that $\dim H(\gamma) = \Phi (\gamma(1- 2^{-\beta}))$. 

\begin{corl}
\label{noucor}
Let $\{S_n \}$ be a dyadic martingale. Assume that there exists a constant $C>0$ such that 
\begin{equation*}
    \|S_n - S_m\|_\infty \leq C + |n-m| , \quad n,m = 1,2, \ldots
\end{equation*}
For $0 < \eta < 1$ consider the set
\begin{equation*}
    E(\eta) = \{x \in \mathbb{R} : \limsup_{n \to \infty} \frac{S_n (x)}{n} \geq \eta \} .
\end{equation*}
Then $\dim E(\eta) \leq \Phi (\eta)$. 
\end{corl}
\begin{proof}
For positive integers $N$ and $k=0,1, \ldots, N-1$, consider the set
\begin{equation*}
   E_k = E_k (N, \eta) = \{x \in \mathbb{R} : \limsup_{n \to \infty} \frac{S_{Nn + k} (x)}{Nn} \geq \eta \} .  
\end{equation*}
Observe that 
\begin{equation*}
  E (\eta) = \cup_{k=0}^{N-1} E_k . 
\end{equation*}
Write $M=N+ C$ and consider the dyadic martingale $\{T_n \}$ given by 
\begin{equation*}
  T_n = T_n (N,k) = M^{-1} S_{Nn + k} , \quad n=1,2,\ldots
\end{equation*}
Since $\|T_{n+1} - T_n\|_\infty \leq 1$ and 
\begin{equation*}
   E_k =  \{x \in \mathbb{R} : \limsup_{n \to \infty} \frac{T_n (x)}{n} \geq N \eta / M \},
   \end{equation*}
Theorem \ref{t:th2} gives that $\dim E_k \leq \Phi (N \eta / M)$. Taking $N \to \infty$, we deduce that $\dim E(\eta) \leq \Phi (\eta)$ . 
\end{proof}

We now prove Corollary \ref{corl:continuous}. 

\begin{proof}[Proof of Corollary \ref{corl:continuous}]
Let $H(\gamma)$ be the set defined in the statement of Corollary \ref{corl1}. Writing $T_n - T_0$ as sum of $T_k - T_{k-1}$, we observe that $\sup_n 2^{-n \beta} \|T_n \|_\infty < \infty$. Then we deduce that $G(\gamma) \subset H(\gamma)$. So, part (a) of Corollary \ref{corl:continuous} follows from Corollary \ref{corl1}. Part (b) follows from (a). We now prove (c). Consider the dyadic martingale $\{S_n \}$ defined in \eqref{e:T} and observe that $\sup_n \|S_n - S_{n-1}\|_\infty < \infty $.  Given constants $\gamma >0$ and $\delta >0$ to be fixed later, pick $0< \eta < \delta \gamma (1-\gamma)^{-1}$. We will show that at almost every point $x \in \reals$ where  
\begin{equation}
   \label{hipotesi} 
\gamma =  \limsup_{N \to \infty} \frac{1}{N} \#  \{1 \leq k \leq N : 2^{- k \beta} T_k (x) > \delta  \} > 0 ,  
\end{equation}
we have that
\begin{equation}
    \label{conclusio}
     \limsup_{N \to \infty} \frac{1}{N} \#  \{1 \leq k \leq N : 2^{- k \beta} T_k (x) < - \eta  \} > 0 .  
\end{equation}
Fix $x \in \reals$. Consider the sets $A=\{1 \leq k \leq N : 2^{- k \beta} T_k (x) > \delta \}$, $B = \{1 \leq k \leq N : - \eta <  2^{- k \beta} T_k (x) \leq  \delta \}$ and $C = \{1 \leq k \leq N : 2^{- k \beta} T_k (x) \leq - \eta \}$. Assume that \eqref{hipotesi} is satisfied and \eqref{conclusio} does not hold. Using \eqref{parts} we obtain
$$
\limsup_{N \to \infty} \frac{S_N (x)}{N} \geq (1- 2^{- \beta}) (\delta \gamma - \eta (1-\gamma)) = \tau.$$
Hence the set of points $x \in \reals$ where \eqref{hipotesi} is satisfied and \eqref{conclusio} does not hold, is contained in the set $E(\tau)$. By Theorem \ref{t:th2} it has Hausdorff dimension smaller than $1$ and hence, Lebesgue measure zero. This finishes the proof. \end{proof}

\section{Adding $\alpha$-divided differences}
Given $f\in\Lambda_{\alpha} (\reals)$ and $0< \varepsilon < 1$, consider the accumulated $\alpha$-divided difference given by 
\begin{equation}
    \label{Theta}
    \Theta_\varepsilon (f) (x) = \int_{\varepsilon}^1  \frac{f(x+h) - f(x)}{ h^\alpha} \frac{dh}{h} , \quad x \in \reals .
\end{equation}
It is clear that $|\Theta_\varepsilon (f) (x) | \leq \|f \|_\alpha \log (1 / \varepsilon)$ for any $x \in \reals$. It was proved in \cite{CLlN} that $\Theta_\varepsilon (f) $ behaves as a dyadic martingale with bounded increments in the sense that there exists a constant $C = C(\alpha) >0$ and a dyadic martingale $\{S_n \}$ with $\|S_{n+1} - S_n \|_\infty \leq C \|f\|_\alpha$ such that 
\begin{equation}
    \label{acotacio}
   \sup_n  \sup_\varepsilon \{ \| \Theta_\varepsilon (f)  - S_n \|_\infty :  2^{-n-1} \leq \varepsilon \leq 2^{-n} \} < \infty . 
\end{equation}
 See also \cite{LlN}. 

\begin{proof}[Proof of Theorem \ref{thm:continuous}]
Let $\{S_n \}$ be the dyadic martingale satisfying \eqref{acotacio}. Since $\|f\|_\alpha \leq 1$, we have
\begin{equation*}
    |\Theta_\varepsilon (f) (x) - \Theta_\delta (f) (x)| \leq \left| \log \frac{\varepsilon}{\delta}\right| , \quad 0< \varepsilon, \delta < 1. 
\end{equation*}
By \eqref{acotacio}, the martingale $\{S_n\}$ satisfies the assumptions of Corollary \ref{noucor}.  
Observe that if $x \in G(\gamma)$ we have
$$
\limsup_{\varepsilon \to 0} \frac{\Theta_\varepsilon (f) (x)}{\log {\varepsilon}^{-1}} \geq \gamma
 $$
 Hence 
 $$G(\gamma) \subset \{x : \limsup_{n \to \infty} \frac{S_n (x)}{n} \geq \gamma \} . 
 $$
So, the estimate in (a) follows from Corollary \ref{noucor}. Part (b) follows directly from part (a). 
\end{proof}

The previous argument also shows the following result. 
\begin{corl}
\label{cor4}
Let $ 0<\alpha<1$ and  $f\in\Lambda_{\alpha} (\reals)$ with $\|f \|_\alpha \leq 1$.  For $0<\eta<1$ consider the set $E(\eta)$ of points $x \in \reals$ such that
\[
 \limsup_{ \varepsilon \rightarrow0^+}\frac{\Theta_\varepsilon (f) (x)}{\log (1/\varepsilon)}>\eta.
\]
Then $\dim E(\eta) \leq \Phi (  \eta)$.  
\end{corl}

When $\Theta_\varepsilon (f) (x)$ diverges at almost every point, one can find any kind of behavior on sets of Hausdorff dimension $1$. 

\begin{thm}\label{t:th4}
Let $ 0<\alpha<1$ and  $f\in\Lambda_{\alpha} (\reals)$. Let $I \subset \reals$ be an interval such that 
\[
\limsup_{\varepsilon \rightarrow0^+} \Theta_\varepsilon (f) (x) = + \infty
\]
for almost every $x\in I$. Then for any sequence $\{M_k\}_{k=1}^{\infty}$ of real numbers there exists a set $E\subset I $ with $\dim E=1$ such that for any $x\in E$ there exists a decreasing sequence $\{\varepsilon_k\}_{k=1}^{\infty}$ of positive numbers tending to $0$, such that
\[
\int_{\varepsilon_{k}}^{1}\frac{f(x+h)-f(x)}{h^{\alpha}}\,\frac{dh}{h} = M_k,\quad k = 1, 2, \ldots.
\]
\end{thm}
\begin{proof}
The assumption gives that the martingale $\{S_n\}$ satisfying \eqref{acotacio} diverges at almost every point of $I$ as well. Pick a constant $C > \sup_n \sup \{ \| \Theta_\varepsilon (f)  - S_n \|_\infty : 2^{-n-1} \leq  \varepsilon \leq 2^{-n} \} $. One can find a set $E\subset I$ with $\dim E=1$ such that for any $x\in E$ there exist increasing sequences ${n}_k = {n}_k(x)$ and ${m}_k = m_k(x)$ of integers, $n_k < m_k < n_{k+1}$, with $S_{n_k} (x) < M_k-C$ and $S_{m_k} (x) > M_k+C$, $k \geq 1$. Thus
\[
\int_{2^{-n_k}}^{1}\frac{f(x+h)-f(x)}{h^{\alpha}} \,\frac{dh}{h} < M_k,
\]
and 
\[
\int_{2^{- m_k}}^{1} \frac{f(x+h)-f(x)}{h^{\alpha}} \,\frac{dh}{h} > M_k.
\]
Hence for any $k \geq 1$, there exists $\varepsilon_k$ with $2^{-m_k} < \varepsilon_k <   2^{-n_k}$, such that
\[
\int_{\varepsilon_k}^{1} \frac{f(x+h)-f(x)}{h^{\alpha}}\,\frac{dh}{h} = M_k.
\]
\end{proof}

Makarov proved that a Bloch martingale $\{S_n \}$ that diverges almost everywhere must satisfy $$
\Lambda_\varphi \{x \in \reals: \lim_{n \to \infty} S_n (x) = + \infty\} >0, 
$$
where $\Lambda_\varphi $ is the Hausdorff measure associated to the function $\varphi (t) = t \sqrt{\log t^{-1} \log \log \log t^{-1} }$. See \cite{Makarov}. So, we may complement Theorem \ref{t:th4} in the following way. 

\begin{thm}\label{MakarovHolder}
Let $0< \alpha < 1$ and $f \in \Lambda_\alpha (\reals)$ such that 
\begin{equation}
    \label{extremeholder}
    \limsup_{\varepsilon \rightarrow0^+} \Theta_\varepsilon (f) (x) = + \infty  
\end{equation}
for almost every $x \in \reals$. Then the set 
\begin{equation}
    \label{definicioset}
J = \{ x \in \reals :    \lim_{\varepsilon \to 0^+} \int_{\varepsilon}^1  \frac{f(x+h) - f(x)}{ h^\alpha} \frac{dh}{h} = + \infty  \}  
\end{equation}
has Lebesgue measure zero but $\Lambda_\varphi (J) >0$
\end{thm}
\begin{proof}
Let $\{S_n \}$ be the dyadic martingale satisfying \eqref{acotacio}. We have 
$$
J= \{x \in \reals : \lim_{n \to \infty} S_n (x) = + \infty \}
$$
which has Lebesgue measure zero. By assumption for almost every $x \in \reals$,  $\Theta_\varepsilon (f) (x) $ has no limit when $\varepsilon \to 0$. Hence the martingale $\{S_n \}$ diverges almost everywhere and by Makarov's result $\Lambda_\varphi (J) >0$. 
\end{proof}

Let us now deduce Theorem \ref{mean} from Theorem \ref{t:th2}. We start with an auxiliary result. 

\begin{lemma}
\label{Means}
Let $0< \alpha < 1$ and let $f \in \Lambda_\alpha (\reals)$ with $\|f \|_\alpha = 1$. For $0< \delta < 1$ and $0< \gamma < 1$, pick constants $0< \eta < \delta \gamma (1- \gamma)^{-1}$ and $0< \gamma^* < \beta:= \delta \gamma - \eta (1- \gamma)$. Fix $s > \Phi (\beta - \gamma^*)$. Then for almost every ($H^s$) $x \in \reals$ such that
\begin{equation}
    \label{superior}
   \limsup_{\varepsilon \to 0^+} \frac{\sigma\{t \in [\varepsilon,1]: f(x+t) - f(x) > \delta t^\alpha \}}{\log {\varepsilon}^{-1}} > \gamma ,  
\end{equation}
we have 
\begin{equation}
    \label{inferior}
   \limsup_{\varepsilon \to 0^+} \frac{\sigma\{t \in [\varepsilon,1]: f(x+t) - f(x) < - \eta t^\alpha \}}{\log {\varepsilon}^{-1}} > \gamma^* .   
\end{equation}
\end{lemma}

\begin{proof}
Fixed $x \in \reals$ and $\varepsilon >0$, consider the sets $A = \{ t \in [\varepsilon,1]: f(x+t) - f(x) > \delta t^\alpha \}$, $
    B=  \{ t \in [\varepsilon,1]: - \eta t^\alpha <  f(x+t) - f(x) \leq \delta t^\alpha \}$ and 
$ C=  \{ t \in [\varepsilon,1]: f(x+t) - f(x) \leq - \eta t^\alpha \}$. 
Observe that $\sigma(B) \leq \log {\varepsilon}^{-1} - \sigma (A)$. Hence if \eqref{superior} 
is satisfied and \eqref{inferior} does not hold, we have
$$
\limsup_{\varepsilon \to 0} \frac{\Theta_{\varepsilon} (f) (x) }{\log {\varepsilon}^{-1}} > \delta \gamma - \eta (1- \gamma) - \gamma^* = \beta - \gamma^* . 
$$
Then the result follows from Corollary \ref{cor4}. 
\end{proof}

\begin{proof}[Proof of Theorem \ref{mean}]
Fix constants $0< \delta < 1 $ and $0< \gamma <1$ and pick $\eta$, $\beta$ and ${\gamma}^*$ as in Lemma \ref{Means}. Consider the sets
$$
A = \{ x \in \reals : \limsup_{h \to 0}  \frac{ \sigma \{t \in [h, 1]: f(x+t) - f(x) > \delta t^\alpha\} } {\log(1/h)} > \gamma \}
$$
and
$$
B = \{ x \in \reals : \limsup_{h \to 0}  \frac{\sigma \{ t \in [h, 1]: f(x+t) - f(x) < - \eta t^\alpha \} }{\log(1/h)} > \gamma^*\} . 
$$
By Lemma \ref{Means}, the set $A \setminus B$ has Lebesgue measure zero. This finishes the proof.
\end{proof}

 For $0< \alpha < 1$ let $\Lambda_\alpha^*$ be the class of functions $f \in \Lambda_\alpha (\reals)$ for which there exists a constant $C=C(f)>0$ such that for any $x,h \in \reals$ there exists $h^* = h^* (x,h)$ with $h/C \leq h^* \leq Ch$ satisfying
\begin{equation}
    \label{Lambda^*}
    |f(x+ h^*) - f(x)| > C^{-1} |h^*|^\alpha . 
\end{equation}
This condition has appeared in \cite{Baranski} and \cite{BishopPeres} in relation to the problem of computing the Hausdorff dimension of the graph of $f$. If $b>1$, Weierstrass lacunary series $f_{b, \alpha} $ defined in \eqref{eq:Weierstrass} is in $\Lambda_\alpha^*$. See Theorem 2.4 of \cite{Baranski}. It is worth mentioning that if $f \in \Lambda_\alpha^*$ then condition \eqref{superior} or \eqref{inferior} are satisfied at almost every point $x \in \reals$ for certain constants $\delta,  \gamma, \eta, \gamma^*$. Condition \eqref{superior} should also be compared with the notion of mean porosity. See the survey of \cite{Shmerkin}. 

\section{Proof of Theorem \ref{t:contraexemple}}
The proof consists of two parts. First we show the dyadic martingale version of Theorem \ref{t:contraexemple}. Then we approximate the $\alpha$-divided differences by their discrete versions arriving at the continuous statement.
\begin{lemma}\label{l:contraexemple}
Let $0<\beta<1$. Then there exists a dyadic martingale $\{S_n\}$ with $\sup_{n}2^{-n\beta}\|S_n\|_{\infty}<\infty$, 
 such that
\[
\limsup_{n\rightarrow\infty}2^{-n\beta}S_n(x)>0,
\]
and
\[
\liminf_{n\rightarrow\infty}2^{-n\beta}S_n(x) \geq 0
\]
for almost every $x\in \mathbb{R}$. Actually the following uniform version of the last inequality holds: for any $\delta>0$ there exists $n_0\in\mathbb{N}$ such that 
\[
2^{-n\beta}S_n(x) \geq -\delta \;\;\; \textup{for any}\;\; n\geq n_0,\; \; x\in\mathbb{R}.
\]
\end{lemma}
\begin{proof}
It is enough to define $\{S_n\}$ on the unit interval $[0,1)$. It will be constructed via a double induction argument. More precisely, we define a pair of increasing sequences $\{k_{jn}\}_{1\leq n\leq n_j}$ and $M_j,\; n\in\mathbb{Z}_+$ of natural numbers satisfying
\begin{equation}\notag
\begin{split}
&k_{00}+ M_0 \leq k_{01} \leq k_{01} + M_0 \leq \dots\leq k_{0n_0} + M_0 \leq k_{10} \leq k_{10} + M_1\dots\leq k_{1n_1}+M_1\leq\dots\\
&\leq k_{20}\leq\dots \leq k_{(j-1)n_{j-1}}+ M_{j-1}\leq
k_{j0}\leq \dots \leq k_{jn_j} + M_j\leq\dots,
\end{split}
\end{equation}
and a martingale $\{S_m\}$ such that:  \textit{(a)} for any $n\geq0$ there exists $j\geq0$ such that $2^{-m\beta}S_m \geq -2^{-n}$ for $m \geq k_{jn_j}+M_j$, and \textit{(b)} $2^{-m\beta}S_m \geq \frac13$ for at least one number $m$ between $k_{j0}$ and $k_{jn_j}+M_j$ on a large portion of $[0,1)$. We start describing the building block of our construction. \par
\textit{Block construction.}\;\par Given a dyadic interval $J$ with $|J| = 2^{-K}$ and a number  $\delta>0$ we define a \textit{building block} $W(\delta,J)$ as follows.\par
Consider a nested sequence of dyadic subintervals of $J$ that shrinks to its left end-point. In other words, let $J_0:= J$, and, given $J_{k-1}$ define $J_k := J_{k-1}^-$, $k \geq 1$ (where $I^-$ is the left half-interval of $I$). Let $M = M(\delta) := \left[\frac{\log\frac{1}{2\delta}}{(1-\beta)\log 2}\right]+1$, so that
\[
\frac12 \leq 2^{M(1-\beta)}\cdot\delta \leq\frac{1}{2^{\beta}}.
\]
Now let $h_I$ be a (slightly renormalized) Haar function corresponding to a dyadic interval $I$, $h_I(x) = 2\chi_{I^-} - \chi_{I}$, and define
\[
s_{k,J}(x) := \delta\cdot2^{K\beta}\cdot 2^{k}h_{J_k}(x), \quad 0\leq k \leq M.
\]
Since $|J_k| = 2^{-K-k}$, then, clearly, $s_{k,J}$ is a martingale difference of rank $K+k$, and
\begin{equation}\label{e:9}
\|2^{-(K+k)\beta}\sum_{m=0}^ks_{m,J}\|_{\infty} \leq  2\delta 2^{k(1-\beta)} \leq 2^{1-\beta},\quad 0\leq k \leq M.
\end{equation}
On the other hand
\begin{equation}\label{e:10}
2^{-(K+k)\beta}\sum_{m=0}^ks_{m,J} \geq -2^{-k\beta}\cdot\delta \geq -\delta.
\end{equation}
Define
\[
W(\delta,J) := \sum_{k=0}^Ms_{k,J},
\]
and observe that 
\begin{equation}\label{e:11}
2^{-(M+K)\beta}\|W(\delta,J)\|_{\infty} \leq 2^{1-\beta},
\end{equation}
and
\begin{equation}\label{e:1111}
2^{-(M+K)\beta}W(\delta,J)(x) \geq \frac12\left(1-2^{-M}\right),\quad x\in J_M.
\end{equation}
In particular, $|J_M| = 2^{-M}|J|$. To summarize, we have constructed a step function $W(\delta,J)$ supported on $J$ whose values are $-\delta2^{K\beta}$ on $J\setminus J_M$, and $\delta2^{K\beta}(2^M-1)$ on $J_M$. Since $2^{M(1-\beta)}\delta\approx 1$, we have $\delta 2^{K\beta} \approx |J|^{-\beta}2^{M(\beta-1)}$ and therefore $\delta2^{K\beta}(2^M-1)\approx |J_{M}|^{-\beta}$ \par
\textit{Arranging the blocks, first step}. \par
Let $\delta_j := 2^{-j-2},\; j\in\mathbb{Z}_+$. We define a (very lacunary) sequence $k_{mn}$ of numbers in the following way. Put $k_{00} := 0$, $J = [0,1)$, and
\[
S_{M(\delta_0)} := W(\delta_0,J).
\]
Now let $k_{01}$ be such that
\[
2^{-k_{01}\beta}\|S_{M(\delta_0)}\|_{\infty} \leq \frac{\delta_0}{2}.
\]
Then we let 
\begin{equation}\notag
\begin{split}
&S_i := S_{M(\delta_0)},\quad M(\delta_0) \leq i \leq k_{01},\\
&S_{k_{01}+M(\delta_0)} := S_{M(\delta_0)} + \sum_{J\in \mathcal{D}_{k_{01}}}W(\delta_0,J),
\end{split}
\end{equation}
We remind that $\mathcal{D}_i$ is the collection of dyadic intervals of rank $i$.\par
We continue iterating the procedure. To elaborate, assume we defined the numbers $k_{0n}$ and the martingale $S_i$ with $0\leq i \leq k_{0n}+M(\delta_0)$. Then we pick $k_{0(n+1)}$ such that
\[
2^{-k_{0(n+1)}\beta}\|S_{k_{0n}+M(\delta_0)}\|_{\infty} \leq \frac{\delta_0}{2},
\]
and
\begin{equation}\notag
\begin{split}
&S_i := S_{k_{0n}+M(\delta_0)},\quad k_{0n}+M(\delta_0) \leq i < k_{0(n+1)},\\
&S_{k_{0(n+1)}+M(\delta_0)} := S_{k_{0n}+M(\delta_0)} + \sum_{J\in \mathcal{D}_{k_{0(n+1)}}}W(\delta_0,J).
\end{split}
\end{equation}
We repeat the construction until we have $n=n_0 := \left[\frac{\log(1-2^{-M(\delta_0)})}{\log \delta_0}\right]+1$.\par
\textit{Arranging the blocks, second step.}\\
We continue to iterate, now also with respect to the parameter $j$.
Assume that we have defined a sequence of numbers $\{k_{mn}\}_{m=0}^{j-1} = 
\{\{k_{0n}\}_{n=0}^{n_0},\dots,\{k_{(j-1)n}\}_{n=0}^{n_{j-1}}\}$ and a sequence of partial sums $\{S_i\},\;
 i=0,\dots, k_{(j-1)n_{j-1}}+ M(\delta_{j-1})$. We apply the procedure from the previous step, now using $\delta_j$ in place of $\delta_0$. In other words, we fix a number $k_{j0} \geq k_{(j-1)n_{j-1}}$ such that
\[
2^{-k_{j0}\beta}\|S_{k_{(j-1)n_{j-1}}+ M(\delta_{j-1})}\| \leq \frac{\delta_j}{2},
\]
and define $S_i$ for $k_{(j-1)n_{j-1}}+M(\delta_{j-1}) \leq i \leq k_{j0} + M(\delta_j)$ as above. Then we proceed to $k_{j1}$ and so on, until we have $n=n_j = \left[\frac{\log(1-2^{-M(\delta_j)})}{\log \delta_j}\right]+1$ (by our assumptions $m_j = M(\delta_j)\approx j$, and $n_j \approx j2^{j}$).\par
\textit{Behaviour of $\{S_m\}.$}\par
First we claim that $S_i$ satisfies the growth condition, that is 
\[
\sup_i 2^{-i\beta}\|S_i\|_{\infty} \leq 2^{1-\beta}.
\]
Indeed, fix a number $i$ and consider the largest $k_{jn}$ such that $k_{jn} \leq i$. We have two options: \textit{(a)} $k_{jn}+M(\delta_j) < i$, and \textit{(b}) $k_{jn}+M(\delta_j) \geq i$. For the option (a) the martingale just stops until we hit the next number $k_{j(n+1)}$ or $k_{(j+1)0}$, in any case, clearly, $S_i = S_{k_{jn}+M(\delta_j)}$, and we have
\begin{equation}\notag
\begin{split}
&2^{-i\beta}\|S_i\|_{\infty} = 2^{-i\beta}\|S_{k_{jn}+M(\delta_j)}\|_{\infty} \leq 2^{-(k_{jn}+M(\delta_j))\beta}\|S_{k_{jn}+M(\delta_j)}\|_{\infty} \leq\\
&2^{-(k_{jn}+M(\delta_j))\beta}\left(\|S_{m}\|_{\infty} + \|S_{k_{jn}+M(\delta_j)}-S_m\|_{\infty}\right),
\end{split}
\end{equation}
where $m = m(n,j)$ is either $k_{j(n-1)} + M(\delta_j)$, if $n\geq1$, or $k_{(j-1)n_{j-1}} + M(\delta_{j-1})$, if $n=0$. In both cases $k_{jn}$ was chosen in such a way that 
\[
2^{-(k_{jn}+M(\delta_j))\beta}\|S_m\|_{\infty} \leq 2^{-k_{jn}\beta}\|S_m\|_{\infty} \leq \frac{\delta_j}{2}.
\]
On the other hand, by construction we have
\[
S_{k_{jn}+M(\delta_j)}-S_m = \sum_{J\in \mathcal{D}_{k_{jn}}}W(\delta_j,J),
\]
hence $\|S_{k_{jn}+M(\delta_j)}-S_m\|_{\infty} =  \|W(\delta_j,J)\|_{\infty}$ for any $J\in \mathcal{D}_{k_{nj}}$. By our choice of $W(\delta_j,J)$ (see \eqref{e:11}) we have
\[
2^{-(k_{nj}+M(\delta_j))\beta}\|W(\delta_j ,J)\|_{\infty} \leq 2^{1-\beta}.
\]
Option (b) is dealt in the same way, only now we use estimate \eqref{e:9} instead.\par
Next we aim to show that 
\[
\liminf_{i\rightarrow\infty} 2^{-i\beta}S_i(x) \geq 0,\quad x\in [0,1).
\]
Again, it follows from our construction, since the martingale $S_i$ consists of very sparse and independent pieces, and by the choice of $k_{jn}$ we always can consider only the tail end of it. In particular, if $i\geq k_{jn}+M(\delta_j)$ for some $j,n$, then by \eqref{e:10} we have $2^{-i\beta}W(\delta_j,J) \geq -2\delta_j$ for any $J\in \mathcal{D}_{k_{jn}}$, hence using the previous argument we get $2^{-i\beta}S_i \geq -3\delta_j$, which proves the estimate, as well as the last part of the statement.\par
Finally we want to estimate the size of the set $E$ of points $x \in \mathbb{R}$ where $\limsup_{i\rightarrow\infty}2^{-i\beta}S_i(x) \geq \frac15$. Fix a pair of numbers $j\in\mathbb{Z}_+$ and $0\leq n\leq n_j-1$. Since, as before,
\[
2^{-(k_{jn}+M(\delta_j))\beta}\|S_{k_{jn}+M(\delta_j)}\|_{\infty} \geq 2^{-(k_{jn}+M(\delta_j))\beta}\|W(\delta_j,J)\|_{\infty} - \frac{\delta_j}{2}
\]
for any $J\in \mathcal{D}_{k_{jn}}$, we can only consider the respective building block $W(\delta_j,J)$. Now, if $|J| = 2^{-k_{jn}}$, we have seen in \eqref{e:1111} that $2^{-(k_{jn}+M(\delta_j))\beta}W(\delta_j,J) \geq \frac{1}{4}$ on the interval $J_{M(\delta_j)}$ with $|J_{M(\delta_j)}| = 2^{-M(\delta_j)}|J|$. On the other hand, if $I$ is the dyadic interval of the next construction step in $J$, that is $|I| = 2^{-k_{j(n+1)}}, \; I\subset J$, again by \eqref{e:1111} we have  $2^{-(k_{j(n+1)}+M(\delta_j))\beta}W(\delta_j, I) \geq \frac14$ on $I_{M(\delta_j)}$. Denote by $\mathcal{F}(J)$ the set of all such intervals, that is, 
\[
\mathcal{F}(J) = \{I_{M(\delta_j)}\subset I: I\in \mathcal{D}_{k_{j(n+1)}}(J)\},
\]
where $\mathcal{D}_m(J)$ is the collection of dyadic intervals of rank $m$ that lie inside $J$. The intervals in $\mathcal{F}(J)$ are disjoint, and they are uniformly distributed over $J$ (for any $I\in\mathcal{D}_{k_{j(n+1)}}(J)$ recall that $I_{M(\delta_j)}$ is a leftmost dyadic subinterval of $I$ of rank $k_{j(n+1)}+M(\delta_j)$). It follows that

\begin{equation}\label{e:2112}
\begin{split}
&\left|\left(\bigcup_{\mathcal{F}(J)}I_{M(\delta_j)}\right)\setminus J_{M(\delta_j)}\right| = \sum_{I_{M(\delta_j)}\in \mathcal{F}(J),\;I_{M(\delta_j)}\subset J\setminus J_{M(\delta_j)}}|I_{M(\delta_j)}| =\\
&(2-2^{-M(\delta_j)})2^{-M(\delta_j)}|J\setminus J_{M(\delta_j)}|.
\end{split}
\end{equation}
An interval $I'$ is called \textit{$\delta_j$-special}, if there exists a number $0\leq n\leq n_j$ and an interval $J \in \mathcal{D}_{k_{jn}}$ such that $I' = J_{M(\delta_j)}$, that is $I'$ is the left-most dyadic subinterval of $J$ of rank $k_{jn}+M(\delta_j)$. The collection of $\delta_j$-special intervals is denoted by $\mathcal{F}_j$. As before, $|I'|^{\beta}W(\delta_j,J) \geq \frac{1}{4}$ on $I'$, and therefore $|I'|^{\beta}S(I') \geq \frac15$ (where $S(I) := S_{i}(x)$ with $x\in I$ and $i = \log_2|I|^{-1}$). It follows from \eqref{e:2112} that
\[
\left|\bigcup_{I'\in\mathcal{F}_j}I'\right| \geq 1-(1-2^{-M(\delta_j)})^{n_j}.
\]
Therefore the set $F_j$ of points $x\in [0,1)$  where
\[
2^{-(k_{jn}+M(\delta_j))\beta}\sum_{J\in\mathcal{D}_{k_{jn}}}W(\delta_j,J)(x) \leq \frac{1}{4}
\]
for all $n=0,\dots, n_j,$ has small Lebesgue measure, namely
\[
|F_j| \leq (1-2^{-M(\delta_j)})^{n_j} \lesssim \delta_j
\]
by our choice of $n_j$. Hence
\[
\left|\left\{x: 2^{-i\beta}S_i(x) \leq \frac15,\; k_{j0}\leq i\leq k_{jn_j}\right\}\right| \lesssim \delta_j.
\]
Since $\sum_j\delta_j \leq 1$, we see immediately that
\[
\left|\left\{x:\;\limsup_{i\rightarrow\infty}2^{-i\beta}S_i(x) \leq \frac15\right\}\right| = 0.
\]
We make another observation which will be useful later. Given a $\delta_j$-special interval $I'$ consider the dyadic interval $\tilde{I}$ of the same length that lies immediately on the left of $I'$, in other words, if $I' = [i2^{-m}, (i+1)2^{-m})$, then $\tilde{I} := [(i-1)2^{-m}, i2^{-m})$ (if the left end-point of $I'$ is $0$, we put $\tilde{I} := \emptyset$, so the intervals that fall out of $[0,1)$ are discarded). These intervals are called \textit{left-$\delta_j$-special}, and their collection is denoted by $\tilde{\mathcal{F}}_j$. Arguing as above we see that
\begin{equation}\label{e:2111}
\left|[0,1)\setminus\left(\bigcup_{\tilde{I}\in\tilde{\mathcal{F}}_j}\tilde{I}\right)\right| \leq 2(1-2^{-M(\delta_j)})^{n_j}\lesssim \delta_j,
\end{equation}
so that almost every point $x\in [0,1)$ lies in $\bigcup_{\tilde{I}\in\tilde{\mathcal{F}}_j}\tilde{I}$ for infinitely many $j\in\mathbb{Z}_+$.
\end{proof}

Now we are ready to prove Theorem \ref{t:contraexemple}.\par

\begin{proof}[Proof of Theorem \ref{t:contraexemple}] 
Fix $\beta:= 1-\alpha$. Consider the martingale $\{S_n\}$ constructed in Lemma \ref{l:contraexemple}. We can assume $S_0 = 0$. We will define a function $f$ defined in the real line as follows. Let $f(0)=0$. The relation $f(b_n)-f(a_n) := 2^{-n}S_n(I)$ for any $I= [a_n,b_n)\in \mathcal{D}_n$, $n\geq 0$, defines $f$ on the dyadic points of $[0,1]$ and we extend $f$ to non-dyadic points of $[0,1]$ by continuity. Observe that since $S_0 = 0$ we have $f(0)= f(1)=0$. Finally we extend $f$ from $[0,1]$ to the whole real line by periodicity. Let us prove that $f\in \Lambda_{\alpha}$. Fix a point $x\in \mathbb{R}$ and a number $0 < h \leq 1$. We aim to show that $|f(x+h) - f(x)| \leq Ch^{\alpha}$ for some absolute constant $C>0$. There exists an increasing sequence of dyadic-rational points $\{a_k\}_{k\in\mathbb{Z}}$ such that $[a_{k-1},a_{k})\in \mathcal{D}$, $\lim_{k\rightarrow-\infty}a_k = x$, $\lim_{k\rightarrow +\infty} a_k = x+h$, and for any $n\in\mathbb{N}$ there exists at most $4$ dyadic intervals of rank $n$ of the form $[a_{k-1},a_{k})$. In other words, we consider a Whitney decomposition of the interval $[x,x+h)$ with $\{a_k\}$ being the endpoints of the  corresponding dyadic intervals. Given $k\in\mathbb{Z}$ denote by $r_k$ the length of the interval $[a_{k-1},a_{k})$, that is $r_k = a_{k}-a_{k-1}$. Clearly,
\begin{equation}\notag
\begin{split}
&|f(x+h) - f(x)| = \left|\sum_{k\in\mathbb{Z}} \left(f(a_k) - f(a_{k-1})\right)\right| \leq \sum_{k\in\mathbb{Z}} \left|f(a_k) - f(a_{k-1})\right| =\\ &=
\sum_{k\in\mathbb{Z}} r_k|S([a_{k-1},a_{k}))|.
\end{split}
\end{equation}
Since by construction $\sup_n 2^{-n (1 - \alpha)} \|S_n\|_{\infty} < \infty $, and the amount of points $a_k$ generating the dyadic intervals of rank $n$ is bounded, there exists a constant $C= C(\alpha) >0 $ such that 
\[
|f(x+h) - f(x)| \leq C\sum_{n\geq \log_2\frac{1}{h}}2^{-n}2^{n(1-\alpha)} \leq \frac{C}{1- 2^{- \alpha}} h^{\alpha},
\]
so $f$ belongs to the corresponding H\"{o}lder class $\Lambda_{\alpha}$.
Next we show that 
\begin{equation}\label{e:201}
\liminf_{h\rightarrow0^+}\frac{f(x+h)-f(x)}{h^{\alpha}} = 0,\;\; x\in[0,1).
\end{equation}
Fix any $x\in [0,1)$ and an arbitrarily small $\delta>0$. By the last part of Lemma \ref{l:contraexemple} there exists a number $N$ such that $2^{-n(1-\alpha)}S_n(t) \geq -\delta$ for any $n\geq N$ and $t\in [0,1)$. Now fix any $0< h\leq 2^{-N}$, and consider the Whitney decomposition of $[x,x+h)$ as before. Clearly, $r_k \leq 2^{-N}$ for all $k\in\mathbb{Z}$, therefore we have
\begin{equation}\notag
\begin{split}
&f(x+h) - f(x) = \sum_{k\in\mathbb{Z}} \left(f(a_k) - f(a_{k-1})\right)= \sum_{k\in\mathbb{Z}} r_k\frac{f(a_k) - f(a_{k-1})}{r_k}=\\
&\sum_{k\in\mathbb{Z}} r_kS([a_{k-1},a_{k})) = \sum_{k\in\mathbb{Z}} 2^{-n_k}S([a_{k-1},a_{k})) \geq -\delta\sum_{k\in\mathbb{Z}} 2^{-n_k}2^{n_k(1-\alpha)},
\end{split}
\end{equation}
where $2^{-n_k} = r_k$ and $\sup_kr_k\leq h$. Since the numbers $n_k$ do not accumulate (we recall that for any $n$ there are at most four numbers $n_k=n$), it follows that
\[
\sum_{k\in\mathbb{Z}} 2^{-n_k}2^{n_k(1-\alpha)} \leq Ch^{\alpha}
\]
for some absolute constant $C>0$, and \eqref{e:201} follows immediately.\par
It remains to show that for almost every $x \in [0,1]$ we have 
\begin{equation}\label{e:202}
\limsup_{h\rightarrow0^+}\frac{f(x+h)-f(x)}{h^{\alpha}} > \frac{1}{20}. 
\end{equation}
Fix a point $x\in [0,1)$ and a number $N$ such that $2^{-n(1-\alpha)}S_n(t) \geq -\frac{1-\alpha}{40}$ for any $n\geq N$ and $t\in [0,1)$. It follows from \eqref{e:2111} that almost every $x$ belongs to infinitely many left-$\delta_j$-special intervals, in particular there is an increasing sequence $\{j_m(x)\}_{m=0}^{\infty}$ such that $x\in \tilde{I}_m(x)\in\tilde{\mathcal{F}}_{j_m}$ and $|\tilde{I}_m| \leq 2^{-N}$. Now for any $j_m$ we define $h_m$ in such a way that $x+h_m$ is the right end-point of the $\delta_{j_m}$-special interval $I_m$ corresponding to $\tilde{I}_m$. In other words, if $\tilde{I}_m = [(i-1)|\tilde{I}_m|, i|\tilde{I}_m|)$ for some $i\in \mathbb{Z}_+$, then $h_m := (i+1)|\tilde{I}_m| - x$. Since $I_m$ is $\delta_{j_m}$-special, we have $|I_m|^{1-\alpha}S(I_m) \geq \frac15$. Consider a Whitney-type decomposition of $[x,x+h_m)$ generated by $\{a_k\}_{k\in\mathbb{Z}}$ as above. In this case, since $x+h_m$ is dyadic-rational, we assume $a_0 = a_1 =\dots = x+h_m$, also, clearly, $[a_{-1},a_0) = I_m$ and $a_k-a_{k-1}= r_k \leq |I_m|\leq 2^{-N}$ for any $k\leq0$. In particular, $r_kS([a_{k-1},a_{k})) \geq -r_k^{\alpha}\frac{1-\alpha}{40},\; k<0$. We therefore have
\begin{equation}\notag
\begin{split}
&f(x+h_m) - f(x) = \sum_{k\leq0} \left(f(a_k) - f(a_{k-1})\right)= \\
& = r_0\frac{f(a_0) - f(a_{-1})}{r_0} + \sum_{k<0} r_k\frac{f(a_k) - f(a_{k-1})}{r_k}=\\
& = |I_m| S(I_m) + \sum_{k<0} r_kS([a_{k-1},a_{k})) \geq \frac15|I_m|^{\alpha} - \frac{1-\alpha}{40}\sum_{k<0} r_k^{\alpha}.
\end{split}
\end{equation}
Since for any given rank there are at most $4$ dyadic intervals of this rank of  the form $[a_{k-1} , a_k)$, we have 
$$\sum_{k<0} r_k^{\alpha} \leq 4\sum_{n\geq \log_2|I_m|^{-1}}2^{-n\alpha} \leq \frac{4}{1-\alpha}|I_m|^{\alpha} . 
$$ 
Hence
\[
f(x+h_m) - f(x) \geq \frac15|I_m|^{\alpha} - \frac{1}{10}|I_m|^{\alpha} = \frac{1}{10}|I_m|^{\alpha} \geq \frac{1}{20}h_m^{\alpha},
\]
because $h_m \leq |I_m| + |\tilde{I}_m| = 2|I_m|$. This finishes the proof of Theorem \ref{t:contraexemple}.
   
   \end{proof}
   
\section{Proof of Theorem \ref{p:51}}
We will construct the function $f$ via a rarefied (with respect to space variable) and lacunary (with respect to frequency scale variable) wavelet series. In fact it will be an analogue of the classical Weierstrass functions which admits better control over the individual atoms. We start by defining the base wavelet $\varphi\in C^{\infty}(\mathbb{R})$ that satisfies the following conditions
\begin{equation}\notag
\begin{split}
& \supp\varphi= \left[-\frac12,\frac12\right],\;\;
\varphi \equiv 1\;\textup{on}\; \left[-\frac{1}{16},\frac{1}{16}\right],\;\;
\varphi \equiv -1\; \textup{on}\; \left[-\frac{7}{16},-\frac38\right]\cup\left[\frac38,\frac{7}{16}\right],\\
&\int_{\mathbb{R}}x^n\varphi(x)\,dx = 0,\;0\leq n\leq 2.
\end{split}
\end{equation}
It is easy to verify (see e.g. \cite{HT91}) that for any sequence $\{c_{jk}\},\; j\in\mathbb{Z}, k\in\mathbb{Z}_+$, satisfying $ |c_{jk}| \leq 2^{-k\alpha}$, $ k\in\mathbb{N}
$, the function
\[
f := \sum_{j\in\mathbb{Z},k\in\mathbb{N}}c_{jk}\varphi_{jk}, \quad \textup{where}\;\;\varphi_{jk}(t) := \varphi(2^kt-j),
\]
belongs to $\Lambda_{\alpha}$. 

We consider a {\it superlacunary} sequence $k_{n}$ of positive integers that will be defined by induction. We put $k_1:= 1$. We set $k_n \geq k_{n-1}+4$ to satisfy a certain condition \eqref{e:522} that we announce in a few lines. Next we put $c_{jk} := 2^{-k\alpha}$, if $k=k_n$ for some $n\in\mathbb{N}$, and $c_{jk} \equiv0$ otherwise, and we let
\[
f := \sum_{n=1}^{\infty}\sum_{j\in\mathbb{Z}}c_{jk_n}\varphi_{jk_n}.
\]
For any $m\geq 2$ we define
$S_m := \sum_{n=1}^{m-1}\sum_{j\in\mathbb{Z}}c_{jk_n}\varphi_{jk_n}$ and $R_m:= \sum_{n=m}^{\infty}\sum_{j\in\mathbb{Z}}c_{jk_n}\varphi_{jk_n}$ to be the main part and the tail of the series representing $f$. 

Assume we have defined $k_n$ for $n=1\dots m-1$ (and therefore $S_m$) for some $m\geq2$. We pick $k_m$ to satisfy the following conditions:
\begin{equation}\label{e:522}
\begin{split}
&2^{-k_m}\cdot\|S_m'\|_{\infty} \leq \varepsilon 2^{-k_m\alpha}\\
&\sup_{|\theta|\leq 10\cdot2^{-k_m}}|S'_m(t_0+\theta)| \leq \varepsilon,\;\; \textup{for every}\; t_0\;\textup{such that}\; S_m'(t_0) = 0
\end{split}
\end{equation}
for some very small absolute constant $\varepsilon>0$ to be chosen later. Observe that for any $m$ the functions $\varphi_{jk_m}$ have disjoint supports, and there are nested sequences  $\{J^{\pm}_{n}\}$ of intervals of length $\frac{1}{8}2^{-k_n}$ such that $\varphi_{j^{\pm}_nk_n} \equiv \pm 1$ on $J^{\pm}_n$ for some $j^{\pm}_n$ with $n\leq m$. 

Fix any point $x\in\mathbb{R}$. Given $m\in\mathbb{N}$ there exist four numbers $r_{m}^{\pm} = r_{m}^{\pm}(x)$, $\rho_m^{\pm} =\rho_m^{\pm}(x)$,  such that $R_{m}(x+r_m^{+}) = \sup_{2^{-k_m}<t\leq 2^{-k_m+1}}R_m(x+t)$, $R_{m}(x+r_m^{-}) = \inf_{2^{-k_m}<t\leq 2^{-k_m+1}}R_m(x+t)$, and $R_{m}(x-\rho_m^{+}) = \sup_{2^{-k_m}<t\leq 2^{-k_m+1}}R_m(x-t)$, $R_{m}(x-\rho_m^{-}) = \inf_{2^{-k_m}<t\leq 2^{-k_m+1}}R_m(x-t)$. In other words, $x+r_m^{\pm}$ is the maximum/minimum point of the  $2^{-k_m}$-periodic function $R_m$ on the interval $[x+2^{-k_m}, x+2^{-k_m+1}]$ (and $x-\rho_m^{\pm}$ on the interval $[x-2^{-k_m+1},x-2^{-k_m}]$ respectively). Clearly $2^{-k_m}\leq |r_m^{\pm}|,|\rho_m^{\pm}|\leq 2^{-k_m+1}$, and
\[
R_m (x-\rho_m^{\pm}) = R_m(x+r_m^{\pm}) = \pm\sum_{n=m}^{\infty}2^{-k_n\alpha}.
\]

Clearly
\[
\frac{f(x+r^{\pm}_m) - f(x)}{r^{\pm}_m} = \frac{S_m(x+r^{\pm}_m) - S_m(x)}{r^{\pm}_m} + \frac{R_m(x+r^{\pm}_m) - R_m(x)}{r^{\pm}_m} := (I^{\pm}) + (II^{\pm}),
\]
and we have $(II^+) \geq0$, $(II^-)\leq0$ by the definition of $r_m^{\pm}$. Consider the following possible situations: 
\begin{itemize}
\item[\textit{(i)}] For one of the numbers $r_m^{\pm}$ we have
\begin{equation}\label{e:553}
\left|\frac{f(x+r^{\pm}_m) - f(x)}{r^{\pm}_m}\right| \leq 1.
\end{equation}
\item[\textit{(ii)}] We have 
\[
\frac{f(x+r^{+}_m) - f(x)}{r^{+}_m} > 1,\quad \frac{f(x+r^{-}_m) - f(x)}{r^{-}_m} < -1,
\]
or
\[
\frac{f(x+r^{+}_m) - f(x)}{r^{+}_m} <-1,\quad \frac{f(x+r^{-}_m) - f(x)}{r^{-}_m} > 1.
\]
\item[\textit{(iii)}] For both $r_m^{\pm}$
\begin{equation}\label{e:554}
\frac{f(x+r^{\pm}_m) - f(x)}{r^{\pm}_m} > 1
\end{equation}
or
\begin{equation}\notag
\frac{f(x+r^{\pm}_m) - f(x)}{r^{\pm}_m} <-1. 
\end{equation}

\end{itemize}
\textit{Case $\textit{(i)}$}. Assume that the inequality holds, say, for $r^+_m$. We claim that in this case
\begin{equation}
\label{nova}
\frac{f(x+r^{-}_m) - f(x)}{r^{-}_m} \leq -\frac{1}{2}2^{k_m(1-\alpha)}.
\end{equation}
Indeed, by \eqref{e:522} the sequence $\{k_n\}$ is chosen in such a way that for $0\leq \theta \leq 2^{-k_m+1}$ we have
\[
\left|S_m(t+\theta)-S_m(t)\right| \leq \int_{0}^{\theta}|S'(t+s)|\,ds\leq 2\varepsilon\cdot 2^{-k_m\alpha}, \quad t\in\mathbb{R}.
\]
 On the other hand, clearly,
\[
\sup_{t\in\mathbb{R}}|R_m(t)| = \sum_{n=m}^{\infty} 2^{-k_n\alpha}  \geq 2^{-k_m\alpha}.
\]
Take $\varepsilon \leq \frac{1}{200}$. It follows immediately that $|S_m(x+r_m^+) - S_m(x+r_m^-)| \leq \frac{1}{50}2^{-k_m\alpha}$, therefore
\begin{equation}\notag
\begin{split}
&\frac{f(x+r^{-}_m) - f(x)}{r^{-}_m} = \frac{r^+_m}{r^-_m}\frac{f(x+r^{+}_m) - f(x)}{r^{+}_m} + \frac{f(x+r^-_m)-f(x+r^+_m)}{r^{-}_m} =\\
&= \frac{r^+_m}{r^-_m}\frac{f(x+r^{+}_m) - f(x)}{r^{+}_m} + \frac{S_m(x+r^-_m)-S_m(x+r^+_m)}{r^{-}_m} + \frac{R_m(x+r^-_m)-R_m(x+r^+_m)}{r^{-}_m}.
\end{split}
\end{equation}
Since $\frac{r^+_m}{r^-_m}\leq 2$, we obtain
\[
\left|\frac{r^+_m}{r^-_m}\frac{f(x+r^{+}_m) - f(x)}{r^{+}_m} + \frac{S_m(x+r^-_m)-S_m(x+r^+_m)}{r^{-}_m}\right| \leq \frac{1}{10}2^{k_m(1-\alpha)}
\]
On the other hand, $R_m(x+r_m^-) - R_m(x+r_m^+) = -2\sum_{n=m}^{\infty}2^{-k_n\alpha} \leq -2^{-k_m\alpha + 1}$, therefore we deduce
\[
\frac{f(x+r^{-}_m) - f(x)}{r^{-}_m} \leq \frac{R_m(x+r^-_m)-R_m(x+r^+_m)}{r^{-}_m} +  \frac{1}{10}2^{k_m(1-\alpha)} \leq -\frac{1}{2}2^{k_m(1-\alpha)}.
\]
This proves \eqref{nova} and we put $h_m := r^{-}_m$ and $h'_m := r^{+}_m$. If \eqref{e:553} is attained at $r^-_m$, we repeat the argument above exchanging $r_m^+$ and $r_m^-$.\par
\textit{Case $(ii)$}.
Clearly there must exist a point $x+\tilde{r}_m$ between $x+r^+_m$ and $x+r^-_m$ such that
\[
f(x+\tilde{r}_m)-f(x) =0,
\]
we immediately put $h'_m := \tilde{r}_m$. On the other hand
\begin{equation}\label{e:5554}
\max \left\{|R_m(x+\tilde{r}_m)- R_m(x+r^+_m)|, |R_m(x+\tilde{r}_m)- R_m(x+r^-_m)| \right\} \geq \sup_{t\in\mathbb{R}}|R_m(t)| \geq 2^{-k_m\alpha}.
\end{equation}
Assume that the maximum is attained at $r_m^+$. Then
\[
f(x+r_m^+)-f(x) = f(x+r_m^+) - f(x+\tilde{r}_m) = S_m(x+r_m^+) - S_m(x+\tilde{r}_m) + R_m(x+r_m^+) - R_m(x+\tilde{r}_m),
\]
and arguing as in the case \textit{(i)} we have
\[
\left|\frac{f(x+r^{+}_m) - f(x)}{r^{+}_m}\right| \geq \frac{1}{2}2^{k_m(1-\alpha)}. 
\]
We then put $h_m:= r^{+}_m$. If the maximum in \eqref{e:5554} is attained at $r_m^-$, we argue similarly.\par
\textit{Case $(iii)$}. Assume we have \eqref{e:554} (the other option is dealt with exactly the same way). Since $R_m(x+r^-_m) - R_m(x) \leq 0$, the arguments above imply that $S_m(x+r^-_m)-S_m(x) \geq r_m^-$. We now show that $R_m(x) \leq0$. Indeed, by our choice of $\{k_n\}$ satisfying \eqref{e:522} the difference $|S_m(x+r^-_m)- S_m(x)|$ is dominated by $2^{-k_m\alpha} \leq -R_m(x+r^-_m)$. Hence the condition $R_m(x)\geq 0$ would immediately imply that $f(x+r^-_m)-f(x)\leq 0$ which contradicts our assumption.

Now we look at the minimum/maximum on the left of $x$. First we claim that both $S_m(x) - S_m(x-\rho_m^{+})$ and $S_m(x) - S_m(x-\rho^-_m)$ are positive. Assume it is not the case, say, for $\rho^+_m$, that is $S_m(x) - S_m(x-\rho_m^{+})\leq0$. Then $S_m'$ should vanish at some point $x+\theta\in [x-\rho^+_m,x+r^-_m]$. By our choice of $k_n$, see \eqref{e:522}, it follows immediately that
\[
\sup_{\theta\approx2^{-k_n}}|S'_m(x+\theta)| \leq \frac{1}{10}.
\]
Therefore
\[
\left|\frac{S_m(x+r^-_m)-S_m(x)}{r^-_m}\right| \leq \frac{1}{10},
\]
and
\[
\frac{f(x+r^{-}_m) - f(x)}{r^{-}_m} \leq \frac{1}{10} + \frac{R_m(x+r^-_m) - R_m(x)}{r^-_m} \leq \frac{1}{10},
\]
so we have a contradiction. This proves that $S_m(x) - S_m(x-\rho_m^{+})\geq0$. A similar argument shows that $S_m(x) - S_m(x-\rho^-_m)\geq0$.\par 
Since $R_m(x) \geq R_m(x-\rho^-_m)$, we obtain
\[
\frac{f(x) - f(x-\rho_m^-)}{\rho_m^-} = \frac{S_m(x) - S_m(x-\rho_m^-)}{\rho_m^-} + \frac{R_m(x) - R_m(x-\rho_m^-)}{\rho_m^-} \geq0.
\]
On the other hand,  since $R_m (x) \leq 0$ we have $R_m(x) - R_m(x-\rho^+_m) \leq -R_m(x-\rho^+_m)$. Hence, as in the previous cases,
\[
\frac{f(x) - f(x-\rho_m^+)}{\rho_m^+} \leq -\frac{1}{2}2^{k_m(1-\alpha)}.
\]
in particular there exists a point $x-\tilde{\rho}_m$ such that $f(x)-f(x-\tilde{\rho}_m)=0$.
We define $h_m:= -\rho_m^+$, and $h'_m := -\tilde{\rho}_m$.

\textbf{Remark.}\; We have constructed a function $f\in\Lambda_{\alpha}$ such that for every $x\in\mathbb{R}$ there exists a couple of sequences $h_m,h'_m$ that satisfy 
\begin{equation}\notag
\begin{split}
&\left|\frac{f(x+h'_m)-f(x)}{h'_m}\right| \leq 1\\
&\frac{|f(x+h_m)-f(x)|}{|h_m |^\alpha} \gtrsim 1.
\end{split}
\end{equation}
It follows from the construction that these two sequences can be chosen in such a way that they both lie on the same side of $x$ (right or left, but it \textit{depends} on the point $x$), but it is not immediately clear that we can fix the side beforehands, i.e. that we can pick such a function $f$ that both $h_m$ and $h'_m$ are, say, positive numbers. One therefore could ask, if for every function $f\in \Lambda_{\alpha}$ there exists at least one point $x$ such that either
\[
\liminf_{\theta\downarrow0}\frac{f(x+\theta)-f(x)}{\theta} = +\infty,
\]
or there exists a finite right derivative of $f$ at $x$.


\section{The Bloch Space}

Let $\mathcal{B}$ be the Bloch space of analytic functions $g$ in the unit disc $\mathbb{D}$ of the complex plane such that $\|g\|_{\mathcal{B}} = \sup_{z \in \mathbb{D}} (1-|z|^2) |g'(z)| < \infty$. Makarov found  a dictionary between Bloch functions and dyadic martingales which has been extremely useful. Given $g \in \mathcal{B}$ consider the dyadic martingale $\{S_n \}$ defined by 
\begin{equation}\label{blochmartingale}
S_n ( \theta) = \lim_{r \to 1} \frac{1}{|I|}\int_I g(r e^{2 \pi i\varphi}) d \varphi, \quad 0 \leq \theta \leq 1, 
\end{equation}
where $I$ is the dyadic interval of generation $n$ which contains $\theta$. It was proved that $\|g\|_* = \sup_n \|S_n - S_{n-1} \|_\infty$ defines an equivalent seminorm in $\mathcal{B}$. See \cite{Makarov}. Makarov also proved several results on the size of the set 
$$
\{\theta : \limsup_{r \to 1} \frac{|g (re^{2 \pi i \theta})|}{\alpha (1-r)} \geq C \}
$$
for various gauge functions $\alpha$. See Lemma 6.11 and the related Lemmas 6.5 and 6.7 in \cite{Makarov}. But these results do not seem to cover the following sharp estimate. 
\begin{corl}\label{Bloch}
Let $g$ be a function in the Bloch space with $\|g\|_* \leq 1$. For $0 < \eta < 1$ consider the set 
\begin{equation*}
    E(\eta)= \{ \theta \in [0,1]: \limsup_{r \to 1} \frac{|g (re^{2 \pi i \theta})|}{\log (1-r)^{-1}} \geq \eta \} . 
\end{equation*}
Then $dim E(\eta) \leq \Phi(\eta)$. 
\end{corl}
\begin{proof}
Consider the dyadic martingale $\{S_n \}$ defined in \eqref{blochmartingale}. Makarov proved the fundamental estimate: for $2^{-n-1} < 1-r \leq 2^{-n}$ we have 
\begin{equation*}
 g(r e^{2 \pi i\theta}) = S_n (\theta) + O(1) , \quad 0 \leq  \theta \leq 1 . 
\end{equation*}
See \cite{Makarov}. Hence Corollary \ref{Bloch} follows from Theorem \ref{t:th2}.

\end{proof}

\end{document}